\documentclass{amsart}

\usepackage{amssymb,amsmath, amsfonts, amsthm}
\usepackage{graphicx}
\usepackage{tikz}
\newtheorem{theorem}{Theorem}[section]
\newtheorem{property}[theorem]{Property}
\newtheorem{observation}[theorem]{Observation}

\newtheorem{lemma}[theorem]{Lemma}

\newtheorem{definition}[theorem]{Definition}
\newcommand{\vertex}{\node[vertex]}
\tikzstyle{vertex}=[circle, draw, inner sep=0pt, minimum size=6pt]

\newcommand{\cartprod}{\hskip1pt \Box \hskip1pt}

\begin{document}
\title{Edgeless graphs are the only universal fixers}
\author{Kirsti Wash\\ Clemson University\\ \texttt{kirstiw@\allowbreak clemson.edu}}
\maketitle
\abstract
Given two disjoint copies of a graph $G$, denoted $G^1$ and $G^2$, and a permutation $\pi$ of $V(G)$, the graph $\pi G$ is constructed by joining $u \in V(G^1)$ to $\pi(u) \in V(G^2)$ for all $u \in V(G^1)$. $G$ is said to be a universal fixer if the domination number of $\pi G$ is equal to the domination number of $G$ for all $\pi$ of $V(G)$. In 1999 it was conjectured that the only universal fixers are the edgeless graphs. Since then, a few partial results have been shown. In this paper, we prove the conjecture completely.
\endabstract

\section{Definitions and Notation}\label{sec1}
We consider only finite, simple, undirected graphs. The vertex set of a graph $G$ is denoted by $V(G)$ and its edge set by $E(G)$. The order of $G$, denoted by $|G|$, is the cardinality of $V(G)$. We will denote the graph consisting of $n$ isolated vertices as $\overline{K_n}$. The {\it open neighborhood} of $v \in V(G)$ is $N(v) = \{\ u \mid \ uv \in E(G)\ \}$, and the {\it open neighborhood} of a subset $D$ of vertices is $N(D) = \displaystyle{\cup_{v\in D}} N(v)$. The {\it closed neighborhood} of $v$ is $N[v] = N(v) \cup \{v\}$, and the closed neighborhood of a subset $D$ of vertices is $N[D] = N(D) \cup D$. A set $S \subseteq V(G)$ is a {\it $2$-packing} of $G$ if $N[x] \cap N[y]=\emptyset$ for every pair of distinct vertices $x$ and $y$ in $S$.

Given two sets $A$ and $B$ of $V(G)$, we say $A$ {\it dominates} $B$ if $B \subseteq N[A]$, and a set $D \subseteq V(G)$ dominates $G$ if $V(G) = N[D]$. The domination number, denoted $\gamma(G)$, is the minimum cardinality of a dominating set of $G$. A $\gamma$-set of $G$ is a dominating set of $G$ of cardinality $\gamma(G)$. 

Given a graph $G$ and any permutation $\pi$ of $V(G)$, the {\it prism of $G$ with respect to $\pi$} is the graph $\pi G$ obtained by taking two disjoint copies of $G$, denoted $G^1$ and $G^2$, and joining every $u \in V(G^1)$ with $\pi(u) \in V(G^2)$. That is, the edges between $G^1$ and $G^2$ form a perfect matching in $\pi G$. For any subset $A \subseteq V(G)$, we let $\pi(A) = \cup_{v \in A} \pi(v)$.

 If $\pi$ is the identity $\bold{1}_G$, then $\pi G \cong G \cartprod  K_2$, the {\it Cartesian product} of $G$ and $K_2$. The graph $G \cartprod K_2$ is often referred to as the {\it prism of} $G$, and the domination number of this graph has been studied by Hartnell and Rall in \cite{RefWorks:8}. 

One can easily verify that $\gamma(G) \le \gamma(\pi G) \le 2\gamma(G)$ for all $\pi$ of $V(G)$. If $\gamma(\pi G) = \gamma(G)$ for some permutation $\pi$ of $V(G)$, then we say $G$ is a {\it $\pi$-fixer}. If $G$ is a $\bold{1}_G$-fixer, then $G$ is said to be a {\it prism fixer}. Moreover, if $\gamma(\pi G) = \gamma(G)$ for all $\pi$, then we say $G$ is a {\it universal fixer}.

In 1999, Gu \cite{RefWorks:14} conjectured that a graph $G$ of order $n$ is a universal fixer if and only if $G = \overline{K_n}$. Clearly if $G=\overline{K_n}$, then for any $\pi$ of $V(G)$ we have $\gamma(\pi G)=n=\gamma(G)$. It is the other direction, the question of whether the edgeless graphs are the only universal fixers, that is far more interesting and is the focus of this paper. Over the past decade, it has been shown that a few classes of graphs do not contain any universal fixers. In particular, given a nontrivial connected graph $G$, Gibson \cite{RefWorks:7} showed that there exists some $\pi$ such that $\gamma(G) \ne \gamma(\pi G)$ if $G$ is bipartite. Cockayne, Gibson, and Mynhardt \cite{RefWorks:11} later proved this to be true when $G$ is claw-free. Mynhardt and Xu \cite{RefWorks:15} also showed if $G$ satisfies $\gamma(G) \le 3$, then $G$ is not a universal fixer. Other partial results can be found in \cite{RefWorks:9, RefWorks:13}. The purpose of this paper is to prove Gu's conjecture, which we state as the following theorem.

\begin{theorem}\label{conj}
A graph $G$ of order $n$ is a universal fixer if and only if $G = \overline{K_n}$.
\end{theorem}

Although the following observation is stated throughout the literature, we give a short proof here for the sake of completeness. 

\begin{observation}\label{disconnect}
Let $G$ be a disconnected graph that contains at least one edge. If $G$ is a universal fixer, then every connected component of $G$ is a universal fixer.
\end{observation}

\begin{proof}
Let $G$ be a disconnected graph containing at least one edge, and let $C_1, \cdots, C_k$ represent the connected components of $G$ where $k \ge 2$. Suppose, for some $j \in \{1, \cdots, k\}$, that $C_j$ is not a universal fixer. There exists a permutation $\pi_j: V(C_j) \to V(C_j)$ such that $\gamma(\pi_j C_j) > \gamma(C_j)$. Now define $\pi:V(G) \to V(G)$ by
\[\pi(x) = \begin{cases} x & \text{if $x \in V(G)\backslash V(C_j)$}\\ \pi_j(x) & \text{if $x \in V(C_j)$}.\end{cases}
\]
Note that $\pi G$ is a disconnected graph which can be written as the disjoint union
\[\left(\bigcup_{i \ne j} C_i \cartprod K_2\right) \cup \ \pi_j C_j.\]

Thus, 
\begin{eqnarray*}
\gamma(\pi G) &=& \gamma\left(\bigcup_{i \ne j} C_i \cartprod K_2\right) +\gamma(\pi_j C_j)\\
&>& \sum_{i\ne j} \gamma\left(C_i \cartprod K_2\right) + \gamma(C_j)\\
&\ge& \gamma(G).
\end{eqnarray*}
Therefore, if there exists a permutation $\pi$ of a connected component $C_j$ of $G$ such that $C_j$ is not a $\pi$-fixer, then $G$ is not a universal fixer. The result follows.
\end{proof}

This observation along with the results of Mynhardt and Xu \cite{RefWorks:15} allow us to consider only nontrivial connected graphs with domination number at least 4. Therefore, we focus on proving the following slightly more specific version of Theorem \ref{conj}.

\begin{theorem}\label{gu}
A nontrivial connected graph $G$ of order $n$ with $\gamma(G) \ge 4$ is a universal fixer if and only if $G = \overline{K_n}$.
\end{theorem}

The remainder of the paper is organized as follows. Section 2 is dedicated to previous results that will be useful in the proof of Theorem \ref{gu}. The proof of Theorem \ref{gu} is given in Section 3.

\section{Known Results}\label{sec2}

In order to study $\pi$-fixers, we will make use of the following results.

\begin{lemma}\label{separable}\cite{RefWorks:15}
Let $G$ be a connected graph of order $n \ge 2$ and $\pi$ a permutation of $V(G)$. Then $\gamma(\pi G) = \gamma(G)$ if and only if $G$ has a $\gamma$-set $D$ such that
\begin{enumerate}
\item[(a)] $D$ admits a partition $D=D_1 \cup D_2$ where $D_1$ dominates $V(G)\backslash D_2$;
\item[(b)] $\pi(D)$ is a $\gamma$-set of $G$ and $\pi(D_2)$ dominates $V(G)\backslash \pi(D_1)$.
\end{enumerate}
\end{lemma}

Note that if a graph $G$ is a universal fixer, then $G$ is also a prism fixer. So applying Lemma \ref{separable} to the permutation $\bold{1}_G$, we get the following type of $\gamma$-set.

\begin{definition}
A $\gamma$-set $D$ of $G$ is said to be {\it symmetric} if $D$ admits a partition $D=D_1 \cup D_2$ where 
\begin{enumerate}
\item[1.] $D_1$ dominates $V(G)\backslash D_2$, and
\item[2.] $D_2$ dominates $V(G)\backslash D_1$.
\end{enumerate}
We write $D=[D_1,D_2]$ to emphasize properties 1 and 2 of this partition of $D$.
\end{definition}

The following two results were shown by Hartnell and Rall \cite{RefWorks:8}, where some statements are in a slightly different form.

\begin{lemma}\label{symmetric}\cite{RefWorks:8}
If $D=[D_1, D_2]$ is a symmetric $\gamma$-set of $G$, then 
\begin{enumerate}
\item[(a)] $D$ is independent.
\item[(b)] $G$ has minimum degree at least $2$.
\item[(c)] $D_1$ and $D_2$ are maximal $2$-packings of $G$.
\item[(d)] For $i \in \{1,2\}$, $\sum_{x \in D_i} \deg x = |V(G)|- \gamma(G)$.
\end{enumerate} 
\end{lemma}

\begin{theorem}\cite{RefWorks:8}
The conditions below are equivalent for any nontrivial, connected graph $G$.
\begin{enumerate}
\item[(a)] $G$ is a prism fixer.
\item[(b)] $G$ has a symmetric $\gamma$-set.
\item[(c)] $G$ has an independent $\gamma$-set $D$ that admits a partition $D=[D_1,D_2]$ such that each vertex in $V(G)\backslash D$ is adjacent to exactly one vertex in $D_i$ for $i \in\{1,2\}$, and each vertex in $D$ is adjacent to at least two vertices in $V(G)\backslash D$.
\end{enumerate}
\end{theorem} 

We shall add to this terminology that if a symmetric $\gamma$-set $D=[D_1,D_2]$ exists such that $|D_1|=|D_2|$, then $D$ is an {\it even} symmetric $\gamma$-set.

\section{Proof of Theorem \ref{gu}}
The proof of Theorem \ref{gu} is broken into three cases depending on the type of symmetric $\gamma$-sets a graph possesses. The following property will be useful in each of these cases.

\begin{property}\label{intersection}
Let $A=[A_1, A_2]$ and $B=[B_1, B_2]$ be symmetric $\gamma$-sets of $G$ such that $|A_1| \le |A_2|$ and $|B_1| \le |B_2|$. 
\begin{enumerate}
\item[(a)] If $|A_1| < |B_1|$, then $A_2 \cap B_1 \ne \emptyset$.
\item[(b)] If $|B_1| = |A_1| < |A_2|$, then $A_2 \cap B_2 \ne \emptyset$.
\end{enumerate}
\end{property}
\begin{proof}
\begin{enumerate}
\item[(a)] By assumption, $|B_1 \backslash A_1| >0$ and $A_1$ dominates $V(G)\backslash A_2$. If $A_2 \cap B_1 = \emptyset$, then by the pigeonhole principle there exists $v \in A_1$ such that $v$ dominates at least two vertices in $B_1$. This contradicts the fact that $B_1$ is a 2-packing. Therefore, $A_2 \cap B_1 \ne \emptyset$. 
\item[(b)] Since $|B_2|  =|A_2| > |A_1|$, replacing $B_1$ with $B_2$ in the above argument gives the desired result. 
\end{enumerate}
\end{proof}

We call the reader's attention to the fact that any universal fixer is inherently a prism fixer. Therefore, in each of the following proofs, we show that for every nontrivial connected prism fixer $G$ there exists a permutation $\alpha$ such that $\gamma(\alpha G) > \gamma(G)$. 

To prove the next three theorems, we introduce the following notation. Let $G$ be a graph and let $\pi$ be a permutation of $V(G)$. For each vertex $v \in V(G)$, we let $v^1$ represent the copy of $v$ in $G^1$ and $v^2$ represent the copy of $v$ in $G^2$. If $A \subseteq V(G)$, we define $A^i = \{\ v^i \mid \ v \in A\ \}$ for $i \in \{1, 2\}$. If $B$ is a set of vertices in the graph $\pi G$, then $B^{(i)} = B \cap V(G^i)$ for $i \in \{1, 2\}$. Additionally, if $B \subseteq V(\pi G)$, we will denote the set of vertices in the original copy of $G$ associated with the vertices of $B$ as $p(B) = \{v \in V(G) \mid \ v^1 \in B \ \text{or} \ v^2 \in B\ \}$.

\begin{theorem}\label{oneven}
Let $G$ be a nontrivial connected prism fixer with $\gamma(G) = 2k$ where $k \ge 2$. If $G$ contains an even symmetric $\gamma$-set $D$ such that $D$ intersects every even symmetric $\gamma$-set of $G$ nontrivially, then $G$ is not a universal fixer.
\end{theorem}

\begin{proof}
Let $D=[D_1,D_2]$ denote the even symmetric $\gamma$-set of $G$ which intersects every even symmetric $\gamma$-set of $G$ nontrivially. By definition, $|D_1|=|D_2| = k$ and by Lemma \ref{symmetric}(c), $D_1$ and $D_2$ are $2$-packings. Label the vertices of $D_1$ as $x_1, x_2, \cdots, x_k$. Since $D_1$ is nonempty and a $2$-packing , there exists a vertex $u_1 \in N(x_1)$ such that $u_1 \notin \bigcup_{i=2}^k N(x_i)$. 

Define the permutation $\alpha$ of $V(G)$ such that for $i = 1, \cdots, k-1$, we have $\alpha(x_i) = x_{i+1}$, $\alpha(x_k) = u_1$ and $\alpha(u_1)=x_1$. Let $\alpha(v)=v$ for all other $v \in V(G) \backslash \left(D_1 \cup \{u_1\}\right)$. Figure \ref{fig:alpha} illustrates $\alpha G$ with this particular permutation.

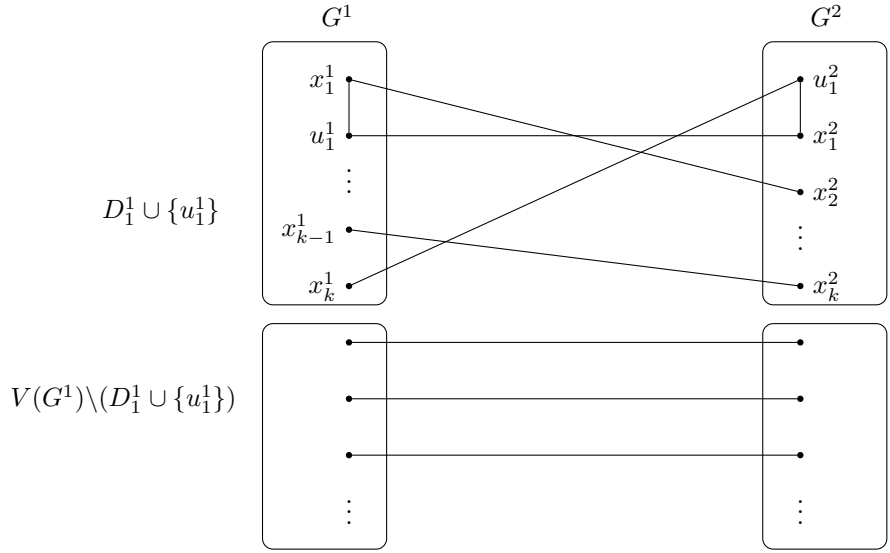
\begin{figure}[h!]
\begin{center}
\begin{tikzpicture}
	 \vertex[fill,minimum size=.07cm] (1a) at (0,0) [label=left:$x_k^1$]{};
	 \vertex[fill, minimum size=.07cm] (2a) at (0,.75) [label=left:$x_{k-1}^1$]{};
	 \vertex[fill, minimum size=.07cm] (3a) at (0,2.75) [label=left:$x_1^1$]{};
	 \vertex[fill, minimum size=.07cm] (5a) at (0, -.75) []{};
	 \vertex[fill, minimum size=.07cm] (6a) at (0, -1.5) []{};
	 \vertex[fill, minimum size=.07cm] (7a) at (0, -2.25) []{};
	 \node (dotss) at (0, -2.9) []{\vdots};
	 \draw (dotss);
	 \node  (dots) at (0,1.5) []{ \vdots};
	 \node (G1) at (-.15, 3.6) []{$G^1$};
	 \draw (G1);
	 \node (G2) at (6.35, 3.6) []{$G^2$};
	 \draw (G2);
 	 \draw (dots);
 	\vertex[fill, minimum size=.07cm] (4a) at (0,2) [label=left:$u_1^1$]{};
 	\vertex[fill, minimum size=.07cm] (1b) at (6,0) [label=right:$x_k^2$]{};
 	\node (V) at (-3,-1.5) []{$V(G^1)\backslash(D_1^1\cup \{u_1^1\})$};
 	\draw (V);
 	\node (dots2) at (6,.75) []{\vdots};
 	\draw (dots2);
	\node (v) at (10,0)[]{};
	\draw (v);
 	\vertex[fill, minimum size=.07cm] (2b) at (6,1.25) [label=right:$x_2^2$]{};
 	\vertex[fill, minimum size=.07cm] (3b) at (6,2) [label=right:$x_1^2$]{};
	 \vertex[fill, minimum size=.07cm] (4b) at (6,2.75) [label=right:$u_1^2$]{};
	 \vertex[fill, minimum size=.07cm] (5b) at (6, -.75) []{};
	 \vertex[fill, minimum size=.07cm] (6b) at (6, -1.5) []{};
	 \vertex[fill, minimum size=.07cm] (7b) at (6, -2.25) []{};
	 \node (dotS) at (6, -2.9) []{\vdots};
	 \draw (dotS);
	 \draw[rounded corners] (-1.15, -3.5) rectangle (.5, -.5);
	 \draw[rounded corners] (-1.15,-.25) rectangle (.5,3.25);
	 \draw[rounded corners] (5.5, -3.5) rectangle (7.15, -.5);
	 \node (D1) at (-2.5, 1) []{$D_1^1 \cup \{u_1^1\}$};
	 \draw (D1);
	 \draw[rounded corners] (5.5,-.25) rectangle (7.15,3.25);
 \path
 	(3a) edge (4a)
	(3b) edge (4b)
	(1a) edge (4b)
	(4a) edge (3b)
	(2a) edge (1b)
	(3a) edge (2b)
	(5a) edge (5b)
	(6a) edge (6b)
	(7a) edge (7b);
\end{tikzpicture}
\caption{$\alpha G$ where $D$ is an even symmetric $\gamma$-set that nontrivially intersects every even symmetric $\gamma$-set of $G$}
\label{fig:alpha}
\end{center}
\end{figure}

Suppose that $\alpha G$ has domination number $2k$ with dominating set $R = R^{(1)} \cup R^{(2)}$. Let $S^1$ be the vertices in $G^1$ that are not adjacent to $R^{(1)}$. Similarly, let $S^2$ be the vertices of $G^2$ that are not adjacent to $R^{(2)}$. Consequently, each vertex of $S^1$ must be dominated by precisely one vertex of $R^{(2)}$, and each vertex of $S^2$ must be dominated by precisely one vertex of $R^{(1)}$ according to $\alpha$. Furthermore, $S^1$ and $S^2$ are $2$-packings, since otherwise there would exist a dominating set of $G$ of order strictly less than $2k$.\\

\begin{enumerate}
\item[Case 1] Assume that $S^1\cap (D_1^1\cup \{u_1^1\}) = \emptyset$ and $R^{(1)}\cap (D_1^1\cup \{u_1^1\}) = \emptyset$.
By assumption, $\alpha(v) = v$ for each $v^1 \in R^{(1)} \cup S^1$. Since $\alpha(S^1) = R^{(2)}$ and $R^{(2)}$ dominates $G^2\backslash \alpha(R^{(1)})$, we have that $p(R) = [p(R^{(1)}), p(R^{(2)})]$ is a symmetric $\gamma$-set of $G$. By symmetry of $\alpha G$, we need only to consider two cases. If $|p(R^{(1)})|=k$, then $p(R)$ is an even symmetric $\gamma$-set. Since $D$ nontrivially  intersects each even symmetric $\gamma$-set of $G$, it follows that $D \cap p(R) \ne \emptyset$. Furthermore, since $\alpha(v)=v$ for each $v^1 \in R^{(1)} \cup S^1$, we know that $D \cap p(R) \subseteq D_2$. Without loss of generality, assume for some $y \in D_2$ that $y^1 \in R^{(1)}$ and $y^2 \in S^2$. This implies that each vertex of $D_1^1$ is either dominated by some vertex $v^1 \in R^{(1)}\backslash \{y^1\}$ or contained in $S^1$. By assumption, $S^1 \cap D_1^1 = \emptyset$ so each vertex of $D_1^1$ is dominated by some vertex $v^1 \in R^{(1)}\backslash \{y^1\}$. However, this contradicts the fact that $D_1^1$ is a $2$-packing since $|R^{(1)}\backslash \{y^1\}| = k-1 < |D_1^1|$. 

So assume $|p(R^{(1)})|<k$. Letting $p(R^{(1)})$ represent $A_1$ and $D_1$ represent $B_1$ in Property~\ref{intersection} (a), we know that $p(R^{(2)}) \cap D_1 \ne \emptyset$. This implies that within $G^2$ we have $R^{(2)} \cap D_1^2 \ne \emptyset$. Furthermore, since $\alpha(S^1) = R^{(2)}$ and $\alpha(v) = v$ for each $v^1 \in S^1$, it follows that $S^1 \cap D_1^1 \ne \emptyset$ within $G^1$, which contradicts our assumption. Therefore, this case cannot occur.

\item[Case 2] Assume that $u_1^1 \in R^{(1)} \cup S^1$.
Let us first suppose that $u_1^1 \in R^{(1)}$. It follows that $x_1^2 \in S^2$ by definition of $\alpha$. This implies that $u_1^2 \in N(R^{(2)})$, since $x_1^2$ is adjacent to $u_1^2$ and $S^2$ is a packing. Moreover, since $x_1$ is the only vertex of $D_1$ that is adjacent to $u_1$, there exists a vertex $v^2 \in R^{(2)}$ that dominates $u_1^2$ within $G^2$ where $\alpha(v) = v$. However, this implies that $v^1 \in S^1$ and there exists an edge between $R^{(1)}$ and $S^1$, because $u_1^1$ is adjacent to $v^1$. This contradicts the fact that $\gamma(G) = 2k$. So we may assume that $u_1^1 \in S^1$. It follows that $x_1^2 \in R^{(2)}$ by definition of $\alpha$. Furthermore, $x_1^1 \in N(R^{(1)})$, since $u_1^1$ is adjacent to $x_1^1$ and $S^1$ is a packing. So there exists a vertex $v^1 \in R^{(1)}$ that dominates $x_1^1$ within $G^1$ where $\alpha(v) = v$. This implies that $v^2 \in S^2$ and there exists an edge between $R^{(2)}$ and $S^2$, which is a contradiction. Therefore, this case cannot occur.

\item[Case 3] Suppose for some $j \in \{2, \cdots, k-1\}$ that $x_j^1 \in R^{(1)} \cup S^1$. We first wish to show that $x_1^1, u_1^1 \in R^{(1)}$. \\
Notice that since $x_j^1 \in R^{(1)} \cup S^1$, it follows that $x_{j+1}^2 \in R^{(2)} \cup S^2$. Now consider $x_j^2$. We claim that $x_j^2$ is not adjacent to any vertex of $R^{(2)}\backslash \{x_j^2\}$. To see this, suppose there exists a vertex $v^2 \in R^{(2)} \backslash \{x_j^2\}$ such that $x_j^2v^2 \in E(G^2)$. We know $v \notin D_1$, since $D_1$ is independent. Therefore, $v \in V(G)\backslash D_1$ where $\alpha(v)=v$. This implies that $x_j^1v^1 \in E(G^1)$, meaning that $R^{(1)} \cup S^1\backslash \{v^1\}$ is a dominating set of $G^1$ of order less than $2k$. This contradiction shows no such $v^2$ exists. Thus, $x_j^2 \in R^{(2)}\cup S^2$, which implies that $x_{j-1}^1 \in R^{(1)} \cup S^1$ according to $\alpha$.

On the other hand, when we try to locate $x_{j+1}^1 \in V(G^1)$, we know $x_{j+1}^1$ is not adjacent to any vertex of $R^{(1)}\backslash \{x_{j+1}^1\}$ by the same reasoning. Thus, $x_{j+1}^1 \in R^{(1)} \cup S^1$. Now, we can apply the same argument inductively to $x_{j+1}^1 \in R^{(1)} \cup S^1$ and $x_{j-1}^1 \in R^{(1)} \cup S^1$ until we arrive at the conclusion that $x_1^1, x_2^1, \cdots, x_k^1 \in R^{(1)} \cup S^1$ and $x_1^2, x_2^2, \cdots, x_k^2 \in R^{(2)} \cup S^2$. Since $x_1^2 \in R^{(2)} \cup S^2$, it follows that $u_1^1 \in R^{(1)} \cup S^1$. But $x_1^1 \in R^{(1)} \cup S^1$ as well and since $x_1^1$ and $u_1^1$ are adjacent, either $x_1^1, u_1^1 \in R^{(1)}$ or $x_1^1, u_1^1 \in S^1$. Furthermore, we can eliminate the case that both are in $S^1$, since $S^1$ is a $2$-packing. Thus, $x_1^1, u_1^1 \in R^{(1)}$. \\

For the remainder of this case, refer to Figure \ref{fig:case}. We know $x_1^2 \in S^2$, since $u_1^1 \in R^{(1)}$. Moreover,  $u_1^2 \in N(R^{(2)})$ since $S^2$ is a $2$-packing, but $u_1^2$ is not in the set $R^{(2)}$ itself since there is no edge between $S^2$ and $R^{(2)}$. In order for $u_1^2$ to be dominated within $G^2$, $u_1^2$ is adjacent to some vertex $v^2 \in R^{(2)}$ such that $v^1 \in S^1$. That is, since $u_1 \notin \bigcup_{i=2}^k N(x_i)$, there exists some vertex $v \in V(G)$ such that $u_1 \in N(v)$ and $\alpha(v)=v$. However, it was assumed that $u_1^1 \in R^{(1)}$. So this would imply that there exists an edge between $R^{(1)}$ and $S^1$, violating $\gamma(G) = 2k$. 
Therefore, this case cannot occur.

\begin{figure}[h]
\centering
\begin{tikzpicture}
\draw[rounded corners] (-2.25, 1.8) rectangle (1.25, 3);
\draw[rounded corners] (4.75, 1.8) rectangle (8.5, 3);
\draw[rounded corners] (-2.25, -1.2) rectangle (1.25, 0);
\draw[rounded corners] (4.75, -1.2) rectangle (8.5, 0);
\vertex[fill, minimum size=.1cm] (x1) at (-1,2.5) [label=below:$x_1^1$]{};
\vertex[fill, minimum size=.1cm] (u1) at (0,2.5) [label=below:$u_1^1$]{};
\node (R1) at (-.5,3.5) []{$R^{(1)}$};
\draw (R1);
\vertex[fill, minimum size=.1cm] (x2) at (6,2.5) [label=right:$x_1^2$]{};
\vertex[fill, minimum size=.1cm] (u2) at (6,1.25) [label=right:$u_1^2$]{};
\node (S2) at (6.5,3.5) []{$S^2$};
\draw (S2);
\vertex[fill, minimum size=.1cm] (v2) at (6,-.5) [label=below:$v^2$]{};
\vertex[fill, minimum size=.1cm] (v1) at (-1,-.5) [label=below:$v^1$]{};
\draw (v1);
\node (S1) at (-.5,.45) []{$S^1$};
\draw (S1);
\node (R2) at (6.5,.45) []{$R^{(2)}$};
\draw (R2);
 \node (alpha) at (2.95,1) []{$\alpha$};
 \draw (alpha);
  \node (a) at (1.75,1.25) []{};
 \node (b) at (4.25,1.25) []{};
 \draw[thick,->] (a) -- (b);
\path
	(x1) edge (u1)
	(x2) edge (u2)
	(u2) edge (v2);
\end{tikzpicture}
\caption{Location of $x_1^1$ and $u_1^1$ in $G^1$}
\label{fig:case}
\end{figure}
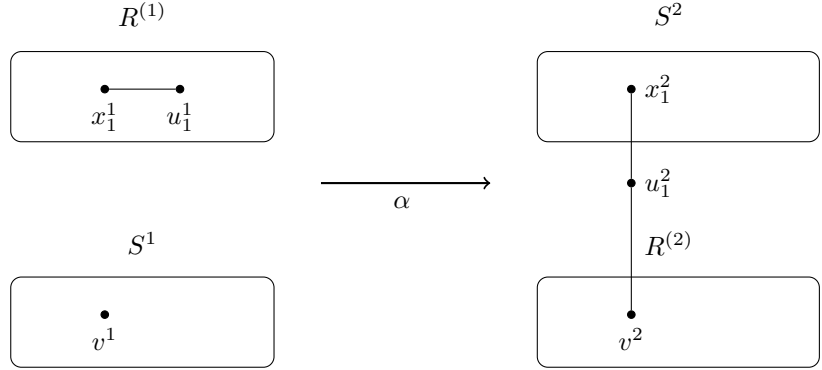

\item[Case 4] Assume that either $x_1^1$ or $x_k^1$ is in $R^{(1)} \cup S^1$. Applying similar arguments as in Case 2 yields the same contradiction. Therefore, this case cannot occur either. 
\end{enumerate}
Thus, no such dominating set $R$ exists for $\alpha G$ and the result follows.
\end{proof}
\vskip5mm

\begin{theorem}\label{odd}
Let $G$ be a nontrivial connected prism fixer with $\gamma(G) = m\ge 4$. If $G$ does not contain an even symmetric $\gamma$-set, then $G$ is not a universal fixer.
\end{theorem}

\begin{proof}
Let $D=[D_1, D_2]$ be any symmetric $\gamma$-set of $G$. Since $D$ is not even, then we may assume $|D_1|>|D_2|$ where $D_1$ and $D_2$ are $2$-packings. Let $|D_1|= k$ where $k < m$, and label the vertices of this set as $x_1, x_2, \cdots, x_k$.  Since $D_1$ is a $2$-packing, there exists a vertex $u_1 \in N(x_1)$ such that $u_1 \notin \bigcup_{i=2}^k N(x_i)$. \\

Define the permutation $\alpha$ of $V(G)$ such that for $i = 1, \cdots, k-1$, $\alpha(x_i) = x_{i+1}$, $\alpha(x_k) = u_1$ and $\alpha(u_1)=x_1$. Let $\alpha(v)=v$ for all other $v \in V(G) \backslash (D_1 \cup \{u_1\})$.\\

As in the proof of Theorem \ref{oneven}, assume that $\alpha G$ has domination number $m$ with dominating set $R = R^{(1)} \cup R^{(2)}$. Let $S^1$ and $S^2$ be defined as in Theorem \ref{oneven} with all the associated properties.\\

Suppose first that neither $R^{(1)}$ nor $S^1$ contain a vertex of $D_1^1$. Note that if $m$ is even and $|R^{(1)}| = \frac{m}{2}$, then $p(R) = [p(R^{(1)}), p(R^{(2)})]$ is an even symmetric $\gamma$-set of $G$, which contradicts our assumption. On the other hand, if $|R^{(1)}| \ne |R^{(2)}|$, then Property \ref{intersection} guarantees that either $R^{(1)}\cap D_1^1 \ne \emptyset$ or $S^1\cap D_1^1 \ne \emptyset$. In either case, we reach a contradiction.

Thus, we need only to consider when $(R^{(1)}\cup S^1)\cap D_1^1 \ne \emptyset$. Similar arguments used in Theorem \ref{oneven}, Cases 2 - 4 complete the proof.
\end{proof}
\vskip5mm

Theorem \ref{odd} implies that if a nontrivial connected universal fixer $G$ with $\gamma(G) \ge 4$ exists, then $G$ contains an even symmetric $\gamma$-set. Moreover, we know from Theorem \ref{oneven} that for each even symmetric $\gamma$-set $D$ of $G$, there exists another even symmetric $\gamma$-set $E$ of $G$ such that $D \cap E = \emptyset$. We now consider graphs that contain at least two pairwise disjoint even symmetric $\gamma$-sets. Note that in this case $\gamma(G)$ is an even integer.

\begin{theorem}\label{multi-even}
Let $G$ be a nontrivial connected prism fixer with $\gamma(G) = 2k$ where $k \ge 2$. If $G$ contains at least two disjoint even symmetric $\gamma$-sets, then $G$ is not a universal fixer.
\end{theorem}

\begin{proof}
Let $D_1, \cdots, D_m$ be a maximal set of pairwise disjoint even symmetric $\gamma$-sets. Since $D_i$ is symmetric, for each $1 \le i \le m$ we can write $D_i =[ X_i , Y_i]$ such that $X_i$ dominates $V(G)\backslash Y_i$ and $Y_i$ dominates $V(G)\backslash X_i$. We let $X = \bigcup_i X_i$.

We know that each $X_i$ is a $2$-packing of size $k$. Thus, we can index the vertices of $X_i$ as $x_{i,1}, x_{i,2}, \cdots, x_{i,k}$ such that $x_{i+1,j}$ is adjacent to $x_{i,j}$ for $1 \le i \le m-1$ and $1 \le j \le k$.

In order to define our permutation of $V(G)$, we first assign an additional index to $X_m$, since we will map $X_m$ to $X_1$. Note that we have already indexed $X_m$ such that $x_{m,j} \in N(x_{m-1,j})$ for $j = 1, \cdots, k$, and this index will be used to map $X_{m-1}$ to $X_m$. Now for $1 \le j \le k$, define $a_j$ such that $x_{m,a_j} \in N(x_{1,j})$, and this index will be used to map $X_m$ to $X_1$. We may define the following permutation of $V(G)$:
\[\alpha(v) = \begin{cases} x_{i+1,j} & \text{if $v=x_{i,j}\ \ \text{for}\  \ 1 \le j \le k\ \ \text{and} \  \ 1 \le i \le m-1$}\\ x_{1,j+1} & \text{if $v=x_{m,a_j}\  \ \text{for}\  \ 1 \le j \le k-1$}\\ x_{1,1} & \text{if $v=x_{m,a_k}$}\\ v & \text{otherwise}.\end{cases}\]

\begin{figure}
\centering
\begin{tikzpicture}
	\draw[rounded corners] (0,0) rectangle (2.5,9.5);
	\vertex (x32) at (1.5,9) [label=left:$x_{3,2}^1$]{};
	\vertex[fill] (x11) at (1.5,8.25) [label=left:$x_{1,1}^1$]{};
	\vertex[rectangle] (x21) at (1.5,7.5) [label=left:$x_{2,1}^1$]{};
	\vertex (x31) at (1.5,6.75) [label=left:$x_{3,1}^1$]{};
	\vertex[fill] (x14) at (1.5,6) [label=left:$x_{1,4}^1$]{};
	\vertex[rectangle] (x24) at (1.5,5.25) [label=left:$x_{2,4}^1$]{};
	\vertex (x34) at (1.5,4.5) [label=left:$x_{3,4}^1$]{};
	\vertex[fill] (x12) at (1.5,3.75) [label=left:$x_{1,2}^1$]{};
	\vertex[rectangle] (x22) at (1.5,3) [label=left:$x_{2,2}^1$]{};
	\vertex (x33) at (1.5,2) [label=left:$x_{3,3}^1$]{};
	\vertex[fill] (x13) at (1.5,1.25) [label=left:$x_{1,3}^1$]{};
	\vertex[rectangle] (x23) at (1.5,.5) [label=left:$x_{2,3}^1$]{};
	\node (G1) at (1.25,10) {$G^1$};
	\draw (G1);
	\draw[rounded corners] (4.5,0) rectangle (7,9.5);
	\node (G2) at (5.75, 10) {$G^2$};
	\draw (G2);
	\vertex (y32) at (5.5,9) [label=right:$x_{3,2}^2$]{};
	\vertex[fill] (y11) at (5.5,8.25) [label=right:$x_{1,1}^2$]{};
	\vertex[rectangle] (y21) at (5.5,7.5) [label=right:$x_{2,1}^2$]{};
	\vertex (y31) at (5.5,6.75) [label=right:$x_{3,1}^2$]{};
	\vertex[fill] (y14) at (5.5,6) [label=right:$x_{1,4}^2$]{};
	\vertex[rectangle] (y24) at (5.5,5.25) [label=right:$x_{2,4}^2$]{};
	\vertex (y34) at (5.5,4.5) [label=right:$x_{3,4}^2$]{};
	\vertex[fill] (y12) at (5.5,3.75) [label=right:$x_{1,2}^2$]{};
	\vertex[rectangle] (y22) at (5.5,3) [label=right:$x_{2,2}^2$]{};
	\vertex (y33) at (5.5,2) [label=right:$x_{3,3}^2$]{};
	\vertex[fill] (y13) at (5.5,1.25) [label=right:$x_{1,3}^2$]{};
	\vertex[rectangle] (y23) at (5.5,.5) [label=right:$x_{2,3}^2$]{};
	\node (D1) at (3.5,-1) {Note that $\alpha(v)=v$ for all other vertices of $G$ not depicted};
	\draw (D1);
	\node (D2) at (-2.75, 4.5) {- $X_1$};
	\draw (D2);
	\vertex[fill] (closed) at (-3.55,4.5) []{};
	\vertex[rectangle] (rec) at (-3.55, 4) []{};
	\vertex (open) at (-3.55,3.5) []{};
	\node (D3) at (-2.75, 4) {- $X_2$};
	\draw (D3);
	\node (D4) at (-2.75, 3.5) {- $X_3$};
	\draw (D4);
	\node (o) at (10, 0) {};
	\draw (o);
	\path
		(x32) edge (x11)
		(x11) edge (x21)
		(x21) edge (x31)
		(x31) edge (x14)
		(x14) edge (x24)
		(x24) edge (x34)
		(x34) edge (x12)
		(x12) edge (x22)
		(x22) edge[bend right=20] (x32)
		(x33) edge (x13)
		(x13) edge (x23)
		(x33) edge[bend left=25] (x23)
		(y32) edge (y11)
		(y11) edge (y21)
		(y21) edge (y31)
		(y31) edge (y14)
		(y14) edge (y24)
		(y24) edge (y34)
		(y34) edge (y12)
		(y12) edge (y22)
		(y22) edge[bend left=20] (y32)
		(y33) edge (y13)
		(y13) edge (y23)
		(y33) edge[bend right=25] (y23)
		(x32) edge (y12)
		(x34) edge (y13)
		(x33) edge (y14)
		(x31) edge (y11)
		(x22) edge (y32)
		(x21) edge (y31)
		(x23) edge (y33)
		(x24) edge (y34)
		(x11) edge (y21)
		(x12) edge (y22)
		(x13) edge (y23)
		(x14) edge (y24);
	
\end{tikzpicture}
\caption{Specific case when $m=3$ and $k=4$}
\label{fig:perm}
\end{figure}

Notice in Figure \ref{fig:perm} that when we consider the indices of $X_m$ as $x_{m,a_j}\in N(x_{1,j})$, we can write the vertices of $X_1$ and $X_m$ as the cycle (in the permutation sense)
\[(x_{m,a_1}, x_{1,2}, x_{m,a_2}, x_{1,3}, \cdots, x_{m,a_k}, x_{1,1})\]
where the following holds for each $1 \le j \le k$:
\begin{enumerate}
\item[(1)] $x_{m,a_j}$ is adjacent to the vertex immediately preceding it within the cycle; and
\item[(2)] $x_{m,a_j}$ is mapped under $\alpha$ to the vertex immediately following it within the cycle.
\end{enumerate}
Furthermore, this cycle cannot be written as a product of subcycles that exhibit the same properties. 

Suppose that $\alpha G$ has domination number $2k$ with dominating set $R = R^{(1)} \cup R^{(2)}$. Let $S^1$ and $S^2$ be defined as in Theorem \ref{oneven} with all the associated properties.

We first claim that $R^{(1)} \cap X^1 \ne \emptyset$.
To see this, suppose neither $S^1$ nor $R^{(1)}$ contains a vertex of $X^1$. By symmetry of $\alpha G$, we need only to consider two cases. If $|R^{(1)}| = k = |R^{(2)}|$, then $p(R) = [p(R^{(1)}),p(R^{(2)})]$ is an even symmetric $\gamma$-set. Since $D_1, \cdots , D_m$ is a maximal set of pairwise disjoint even symmetric $\gamma$-sets, for some $1 \le i \le m$ $D_i \cap p(R) \ne \emptyset$. Moreover, $D_i \cap p(R) \subseteq Y_i$, and without loss of generality we may assume $y_{i,j}^1 \in R^{(1)}$ and $y_{i,j}^2 \in S^2$ for some $1 \le j \le k$. It follows that each vertex of $X_i^1$ is either dominated by some vertex $v^1 \in R^{(1)}\backslash \{y_{i,j}^1\}$ or contained in $S^1$. By assumption, $S^1 \cap X_i^1 = \emptyset$ so each vertex of $X_i^1$ must be dominated by some vertex $v^1 \in R^{(1)} \backslash \{y_{i,j}^1\}$. However, this contradicts the fact that $X_i^1$ is a $2$-packing since $|R^{(1)}\backslash \{y_{i,j}^1\}| = k-1 < |X_i^1|$. Therefore, either $R^{(1)} \cap X^1 \ne \emptyset$ and we are done, or $S^1 \cap X^1 \ne \emptyset$. If $S^1 \cap X^1 \ne \emptyset$, then $R^{(2)} \cap X^2 \ne \emptyset$ by definition of $\alpha$. In this case, simply relabel $G^1$ and $G^2$ so that $R^{(1)} \cap X^1 \ne \emptyset$. 

 On the other hand, if $|R^{(1)}| < k$, then $|S^1 \cap X_i^1| \ne 0$ for each $1 \le i \le m$, since each $X_i$ is a $2$-packing and every vertex of $G^1$ is either in $N[R^{(1)}]$ or in $S^1$. This implies for each $1 \le i \le m$ that $|R^{(2)} \cap X_i^2| \ne 0$ by definition of $\alpha$. As before, simply relabel $G^1$ and $G^2$ so that $|R^{(1)}| \ge k$, and we have $R^{(1)} \cap X^1 \ne \emptyset$.

We next claim that $S^2 \cap X_1^2 \ne \emptyset$.
 From above, we may assume $|R^{(1)}| \ge k$. If $|R^{(1)}| > k$, then $|R^{(2)}| <k$. This implies that $S^2 \cap X_1^2 \ne \emptyset$, since $X_1$ is a $2$-packing and every vertex of $G^2$ is either in $N[R^{(2)}]$ or in $S^2$. So assume that $|R^{(1)}|=k$, and let $x_{i,a}^1 \in R^{(1)}$ for some $1 \le i \le m$ and $1 \le a \le k$. If $i=m$, then by definition of $\alpha$ we have $S^2 \cap X_1^2 \ne \emptyset$. So assume $i \ne m$. Since $Y_i$ is a $2$-packing and no vertex of $Y_i$ is adjacent to a vertex of $X_i$, there exist at least $|R^{(1)} \cap D_i^1|$ vertices in $S^1 \cap Y_i^1$. Moreover, since each vertex of $Y_i$ is mapped to itself under $\alpha$, we know there exist at least $|R^{(1)} \cap D_i^1|$ vertices in $R^{(2)}\cap Y_i^2$. This, together with the fact that $|R^{(1)}|=k=|R^{(2)}|$, gives 
\begin{eqnarray*}
|R^{(2)} \backslash (R^{(2)} \cap Y_i^2)| &\le& k - |R^{(1)} \cap D_i^1|\\
&\le& k-1.
\end{eqnarray*}
Therefore, since $X_i$ is a $2$-packing and each vertex of $G^2$ is either in $N[R^{(2)}]$ or in $S^2$, $S^2 \cap X_i^2 \ne \emptyset$.  So assume $x_{i,b}^2 \in S^2$ for some $1 \le b \le k$. If $i=1$ or if $m=2$, then we are done with the proof of this claim. So assume $m >2$ and $i \not\in \{ 1, m\}$. By definition of $\alpha$, $x_{i-1,b}^1 \in R^{(1)}$. Applying the above argument inductively, eventually we have $S^2 \cap X_1^2 \ne \emptyset$. Let  $r=|S^2 \cap X_1^2|>0$.

We next claim that $r<k$. To see this, suppose that $r = k$. 
Because $X_1$ dominates $V(G)\backslash Y_1$, we have $R^{(2)} = Y_1^2$. However, under $\alpha$ this implies $R^{(1)} = X_m^1$ and $S^1 = Y_1^1$. This contradicts the fact that $Y_1 \subset N(X_m)$, since $R^{(1)}$ was defined to be a set in $G^1$ that dominates all but $S^1$. Thus, we may conclude that $r <k$.

Let $x_{1,b_1}^2, x_{1,b_2}^2, \cdots, x_{1,b_r}^2$ be the vertices of $S^2 \cap X_1^2$. There exist exactly $r$ vertices in $R^{(1)} \cap X_m^1$; call them $x_{m,c_1}^1, x_{m,c_2}^1, \cdots , x_{m,c_r}^1$. We claim for some $x_{1,b_j}^2 \in S^2\cap X_1^2$ that $x_{1,b_j }^1\not\in N(R^{(1)} \cap X_m^1)$. So assume not; that is, assume $\{x_{1,b_1}^1, x_{1,b_2}^1, \cdots, x_{1,b_r}^1 \} \subset N(R^{(1)} \cap X_m^1)$. This implies there exists some relabeling of the $b_j$'s and $c_j$'s such that $x_{m,c_j} \in N(x_{1,b_j})$ and $\alpha(x_{m,c_j}) = x_{1,b_j+1}$ for $b_j \in \{1, \cdots, k-1\}$ and $\alpha(x_{m,c_j}) = x_{11}$ if $c_j = a_k$ where $a_k$ is the index first given to $x_m$ to define $\alpha$. Consequently, there exists some subcycle of $(x_{m,a_1}, x_{12}, x_{m,a_2}, x_{13}, \cdots, x_{m,a_k}, x_{11})$ consisting of the vertices $x_{1,b_1}, x_{1,b_2}, \cdots, x_{1,b_r}, x_{m,c_1}, x_{m,c_2}, \cdots , x_{m,c_r}$ such that for each $1 \le j \le r$:
\begin{enumerate}
\item[(1)] $x_{m,c_j}$ is adjacent to the vertex immediately preceding it within its subcycle; and
\item[(2)] $x_{m,c_j}$ is mapped under $\alpha$ to the vertex immediately following it within its subcycle.
\end{enumerate}

However, this contradicts the construction of $\alpha$ unless $r=k$, which we know to be false. 
Thus, for some $x_{1,b_j}^2 \in S^2 \cap X_1^2$, $x_{1,b_j}^1 \in S^1$ or $x_{1,b_j}^1 \in N[R^{(1)}\backslash(R^{(1)} \cap X_m^1)]$.
\begin{figure}[h!]
\centering
\begin{tikzpicture}
	\node (R) at (3,2) {$R^{(1)}$};
	\draw[rounded corners] (0,0) rectangle (6,1.5);
	\vertex[fill, minimum size=.1cm] (v1) at (1,.75) [label=below:$x_{m,c_1}^1$]{};
	\vertex[fill, minimum size=.1cm] (v2) at (2,.75) [label=below:$x_{m,c_2}^1$]{};
	\vertex[fill, minimum size=.1cm] (v3) at (3,.75) [label=below:$x_{m,c_3}^1$]{};
	\vertex[fill, minimum size=.1cm] (v4) at (4.5,.75) [label=right:$x_{i,d}^1$]{};
	\vertex[fill, minimum size=.1cm] (v5) at (4.5,-.75) [label=below:$x_{1,b_2}^1$]{};
	\node (S) at (10,2) {$S^2$};
	\draw[rounded corners] (7,0) rectangle (13,1.5);
	\vertex[fill, minimum size=.1cm] (w1) at (8,.75) [label=below:$x_{1,b_1}^2$]{};
	\vertex[fill, minimum size=.1cm] (w2) at (9,.75) [label=below:$x_{1,b_2}^2$]{};
	\vertex[fill, minimum size=.1cm] (w3) at (10,.75) [label=below:$x_{1,b_3}^2$]{};
	\vertex[fill, minimum size=.1cm] (w4) at (11.5,.75) [label=right:$x_{i+1,d}^2$]{};
	\vertex[fill, minimum size=.1cm] (w5) at (11.5,-.75) [label=below:$x_{i,d}^2$]{};
	\draw (R);
	\draw (S);
	\path
		(v4) edge (v5)
		(w4) edge (w5)
		(w2) edge (w5);
\end{tikzpicture}
\caption{Specific case when $|S^2 \cap X_1^2|=3$}
\label{fig:4}
\end{figure}
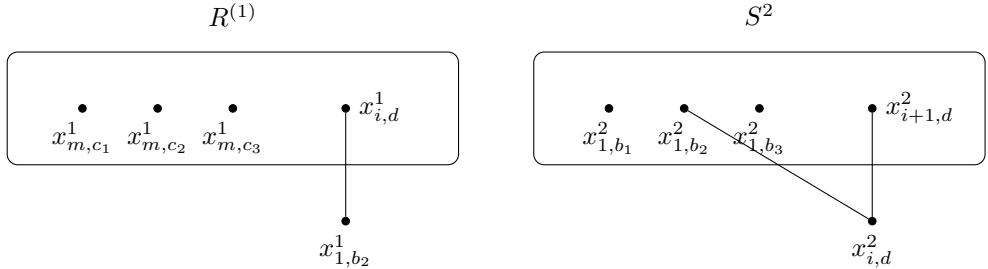
If $x_{1,b_j}^1 \in S^1$, then by definition of $\alpha$, $x_{2,b_j}^2 \in R^{(2)}$. Since $x_{1,b_j}^2 \in N(x_{2,b_j}^2)$, this implies there exists an edge between $R^{(2)}$ and $S^2$. This contradiction shows $x_{1,b_j}^1 \in N[R^{(1)}\backslash (R^{(1)}\cap X_m^1)]$. So assume $v^1 \in R^{(1)}$ where $x_{1,b_j}^1 \in N[v^1]$. If $\alpha(v)=v$, then $v^2$ and $x_{1,b_j}^2$ are both in $S^2$, which contradicts $S^2$ being a $2$-packing. On the other hand, if $\alpha(v) \ne v$, then $v = x_{i,d}$ for some $i \ne m$ and $1 \le d \le k$. 
\begin{enumerate}
\item[Case 1] Assume that $i=1$. Since $X_i$ is a $2$-packing, it follows that $v^1 = x_{1,b_j}^1 \in R^{(1)}$. Thus, $x_{2,b_j}^2 \in S^2$ by definition of $\alpha$. But $x_{1,b_j}^1$ was assumed to be in $S^2$, so this violates $S^2$ being a $2$-packing. Therefore, this case cannot occur.
\item[Case 2] Assume that $i \not\in\{ 1,m\}$. Immediately this implies that $m >2$. Furthermore, $\alpha(x_{i,d})=x_{i+1,d}$, and we have $x_{i,d}^2 \in N(x_{1,b_j}^2) \cap N(x_{i+1,d}^2)$, which contradicts $S^2$ being a $2$-packing, as shown in Figure \ref{fig:4}. Thus, this case cannot occur either.
\end{enumerate}
Having considered all cases, we have shown such a dominating set $R=R^{(1)} \cup R^{(2)}$ of $\alpha G$ does not exist of order $2k$. Hence, the result follows.
\end{proof}

In conclusion, we have shown any nontrivial connected prism fixer $G$ with $\gamma(G) \ge 4$ is not a universal fixer. Moreover, if a graph is not a prism fixer, then it cannot be a universal fixer. Therefore, Theorem \ref{gu} follows. Finally, by Observation \ref{disconnect} and previous results found in \cite{RefWorks:15}, we know that all graphs containing at least one edge are not universal fixers. That is, the only universal fixers that exist are the edgeless graphs.
\section*{Acknowledgments}
I would like to thank Wayne Goddard for several useful suggestions in the write-up of Theorem~\ref{multi-even}. I would also like to thank Doug Rall for the many discussions we had over the implications of Property~\ref{intersection}.

\end{document}